\documentclass[amssymb,twoside,11pt,amscd,reqno]{amsart}

\usepackage{latexsym,amsfonts}
\usepackage{amsmath,amsthm}          
\usepackage{euscript}                
\usepackage{epsfig}
\usepackage[all]{xy}

\theoremstyle{plain} \numberwithin{equation}{section}
\newtheorem{thm}{Theorem}[section]

\newtheorem{prop}[thm]{Proposition}

\theoremstyle{definition}
\newtheorem{remark}{Remark}[section]

\newtheorem{defn}[remark]{Definition}
\newtheorem{ex}[remark]{Example}

\newcommand{\bi}{\begin{itemize}}
\newcommand{\ei}{\end{itemize}}
\newcommand{\bp}{\begin{proof}}
\newcommand{\ep}{\end{proof}}

\begin{document}

\title{The K\"ahler rank of compact complex manifolds}

\author[Chiose]{Ionu\c{t} Chiose$^{\ast}$}
\thanks{$\ast$ Supported by a Marie Curie International Reintegration
Grant within the $7^{\rm th}$ European Community Framework Programme and CNCS grant PN-II-ID-PCE-2011-3-0269}

\date{}

\begin{abstract}
The K\"ahler rank was introduced by Harvey and
Lawson in their 1983 paper 
as a measure of the {\it k\"ahlerianity} of a compact complex surface. In this work we generalize this notion to the case of compact complex manifolds and we prove several results related to this notion. 
We show that on class $VII$ surfaces, there is a correspondence between the closed positive forms on a surface and those on a blow-up in a point. We also show that a manifold of maximal K\"ahler rank which satisfies an additional condition is in fact K\"ahler.\\
{\em Mathematics Subject Classification (2010)} 32J27 (primary), 32J15 (secondary)

\end{abstract}

\maketitle

\section*{Introduction}






In \cite{blaine}, Harvey and Lawson introduced the K\"ahler rank of a compact complex surface, a quantity intended to
measure how far a surface is from being K\"ahler. A surface has K\"ahler rank $2$ iff it is K\"ahler. 
It has K\"ahler rank $1$ iff it is not K\"ahler but still admits a closed (semi-) positive $(1,1)$-form whose zero-locus is contained in a curve. In the remaining cases, it has K\"ahler rank $0$.


In this paper we generalize the notion of K\"ahler rank to compact complex manifolds of arbitrary dimension and study its
properties.

First, we discuss the problem of the bimeromorphic invariance of the K\"ahler rank. There are examples that show that it is not a bimeromorphic invariant. However, two bimeromorphic surfaces have the same K\"ahler rank \cite{chiosetoma}. This was shown by classifying the surfaces of rank $1$. In this paper we take a different approach, local in nature, which was alluded to in \cite{chiosetoma}. Namely, we study the problem of when a plurisubharmonic function on the blow-up is the pull-back of a smooth function. However, this method leads to an involved system of differential equations, and we were able to solve this system only up to order $3$. Thus we obtain:

\begin{thm}\label{system2}
Let $X$ be a compact, complex, non-K\"ahler surface with $b_1(X)=1$, and let $p:X'\to X$ be the blow-up of $X$ at a point. Suppose that $\omega'$ is a closed, positive $(1,1)$ form on $X'$. Then there exists $\omega$ a closed positive $(1,1)$ form on $X$ of class ${\mathcal C}^1$ such that $p^*\omega=\omega'$.
\end{thm}

Second, we study the manifolds of maximal K\"ahler rank, i.e., those manifolds that admit a positive $d$-closed $(1,1)$-form of strictly positive volume. It is conjectured that such manifolds are in the Fujiki class ${\mathcal C}$. Under an additional condition, we prove that they are in fact K\" ahler:

\begin{thm}
Let $X$ be a compact complex manifold of dimension $n$ such that there exists $\{\alpha\}\in H^{1,1}_{BC}(X,{\mathbb R})$ a nef class such that $$\int_X\alpha^n>0$$ Suppose moreover that there exists $h$ a Hermitian metric on $X$ such that $$i\partial\bar\partial h=0, \partial h\wedge\bar\partial h=0$$ Then $X$ is K\"ahler.
\end{thm}

The same method yields a simpler proof of a key theorem of Demailly and P\u aun in \cite{demaillypaun}.

\section*{Acknowledgements} We would like to thank Radu Alexandru Todor for his help with the proof of Proposition \ref{radu}.

\section{Definition and examples}

The K\"ahler rank of a manifold is the maximal rank a closed positive $(1,1)$-form can reach on the manifold:

\begin{defn}
Let $X$ be a compact complex manifold of dimension $n$. The K\"ahler rank of $X$, denoted $Kr(X)$, is
\begin{equation}
Kr(X)=\max\left\{k\vert \exists\omega\in {\mathcal C}^{\infty}_{1,1}(X,{\mathbb R}),\omega\geq 0, d\omega=0, \omega^k\neq 0\right\}
\end{equation}
\end{defn}

The original definition in \cite{blaine} for surfaces required that the form $\omega$ appearing in the definition have zeroes in a analytic subset of $X$. 
Corollary 4.3 in \cite{chiosetoma} shows that the definition above coincides with the one in \cite{blaine} for surfaces.

\begin{remark}
Note that if $Kr(X)=\dim X$ then for every $p\in\overline{0,n}$ the operator $\partial :H^{p,0}(X)\to H^{p+1,0}(X)$ is zero, while, if $Kr(X)=0$ then $\partial :H^{1,0}(X)\to H^{2,0}(X)$ is into. Indeed, if $\sigma\in H^{1,0}(X)\setminus\{0\}$ satisfies $\partial\sigma=0$, then $i\sigma\wedge\bar\sigma$ is a closed, non-zero positive $(1,1)$-form.
\end{remark}

\begin{remark}
As in the surface case considered in \cite{blaine}, on a compact complex manifold $X$ of K\"ahler rank $Kr(X)=k$, there exists a complex analytic canonical foliation ${\mathcal F}$ of codimension $k$. It is defined on the open set
\begin{equation}
{\mathcal B}=\{x\in X\vert\exists\omega\in {\mathcal C}^{\infty}_{1,1}(X,{\mathbb R}), d\omega=0, \omega\geq 0, \omega^k(x)\neq 0\}
\end{equation}
and is characterized by $\omega^k\vert{\mathcal F}=0,\forall\omega\geq 0,d\omega=0$.
\end{remark}

\begin{ex}
A compact complex surface $X$ has K\"ahler rank $2$ if and only if it is K\"ahler (see remark \ref{kahsurf} below) and this is equivalent to $b_1(X)$ even (see \cite{lamari}). When $b_1(X)$ is odd but at least $3$, then $H^{1,0}(X)\neq 0$ and if $\sigma$ is a non-zero holomorphic $1$-form on $X$ then it is $d$-closed, hence $i\sigma\wedge\bar\sigma$ is a $d$-closed positive $(1,1)$-form on $X$. If $b_1(X)=1$, then the main results of \cite{chiosetoma} and \cite{brunella} show that the only surfaces of K\"ahler rank equal to $1$ are the Inoue surfaces and some Hopf surfaces. The other known surfaces (the other Hopf surfaces and the Kato surfaces) have K\"ahler rank $0$.
\end{ex}

\begin{ex}
In \cite{hironaka} the author constructed an example of a $3$-fold $X$ which is a proper modification of a K\"ahler manifold but which is not K\"ahler. In fact, it
is a proper modification $p:X\to {\mathbb P}^3$ of the projective space. One can take $p^*\omega_{FS}$, where $\omega_{FS}$ is the Fubiny-Study metric, to obtain a closed positive $(1,1)$-form, not everywhere degenerate, on a manifold that is not K\"ahler. Therefore, unlike the surface case, in higher dimensions there are manifolds of maximal K\"ahler rank and which are not K\"ahler.
\end{ex}

\begin{ex}
The well-known Iwasawa $3$-fold is the quotient $H/\Gamma$ where $H$ is the group of matrices of the form 
\[ \left( \begin{array}{ccc}
1 & x & z \\
0 & 1 & y \\
0 & 0 & 1 \end{array} \right)\] 
with complex entries, and $\Gamma$ is the subgroup of the matrices whose entries have integer real and imaginary entries. Then the holomorphic $1$-forms on $dx$, $dy$ and $dz-xdy$ on $H$ induce three holomorphic $1$-forms on $H/\Gamma$ denoted by $\sigma_1$, $\sigma_2$ and $\sigma_3$ respectively. Then $d\sigma_3=-\sigma_1\wedge\sigma_2$, hence $\sigma_3$ is not $d$-closed, therefore $Kr(H/\Gamma)\leq 2$. But $\sigma_1$ and $\sigma_2$ are $d$-closed, therefore the form $\omega=i\sigma_1\wedge\bar\sigma_1+i\sigma_2\wedge\bar\sigma_2$ is closed and positive, and $\omega^2\neq 0$, therefore the Kahler rank is $2$.
\end{ex}

\begin{ex}
In \cite{oguiso} the author constructed a Moishezohn $3$-fold $Y$ that contains an algebraic $1$-cycle 
${\ell }$ homologous to zero and which moves and covers the whole $Y$. Such a manifold cannot have maximal K\"ahler rank. Indeed, if $\omega$ is a closed positive $(1,1)$-form on $Y$, and if $y\in Y$ is arbitrary, let ${\ell}'$ be a $1$-cycle passing through $y$ and which is homologous to zero. Then $$\int_{{\ell}'}\omega=0$$ and therefore at $y$, $\omega$ cannot have rank $3$. Therefore $\omega^3=0$. This example shows that for dimension at least $3$ the K\"ahler rank is not a bimeromorphic invariant. However, it is expected that, if $Y\to X$ is the blow-up of a compact complex manifold $X$ in a point, then $Kr(X)=Kr(Y)$.
\end{ex}

\begin{ex}
In \cite{fuliyau} the authors constructed a complex structure on  the connected sum
${\#}_kS^3\times S^3$ of $k\geq 2$ copies of $S^3\times S^3$ and a banced metric $g^2$
which is $i\partial\bar\partial$-exact. Such a manifold has K\"ahler rank equal to $0$. Indeed, if
$\omega$ is a closed positive $(1,1)$-form, then its trace with respect to $g^2$ is zero, hence
the form $\omega$ has to be $0$.
\end{ex}

\begin{remark}
Starting with the above examples, and taking products, one can obtain compact complex manifolds of any dimension $n\geq 2$ and any K\"ahler rank $0\leq Kr\leq n$.
\end{remark}

\section{The bimeromorphic invariance of the K\"ahler rank for class $VII$ surfaces}

In this section we discuss the bimeromorphic invariance of the K\"ahler rank on class $VII$ surfaces, the only non-trivial case. We show that the problem can be reduced to a system
of differential equations, and then we solve the system up to order $3$, thus proving theorem \ref{system2}

\subsection{Preliminaries}

Suppose $X$ is a surface with $b_1=1$ and let $\pi: X'\to X$ be the
blow-up of $X$ in a point $p$. Let $\gamma^{0,1}$ be a $\bar\partial$ closed $(0,1)$ form on $X$ which generates $H^{0,1}(X)$. Then $\gamma'^{0,1}=\pi^*\gamma^{0,1}$ generates $H^{0,1}(X')$.

Let $\omega'$ be a closed, positive $(1,1)$ form on $X'$; then it is $d$ exact \cite{blaine}, Proposition 37. We want to show that there exists $\omega$ on $X$ such that $\pi^*\omega=\omega'$. Then on $X'$, $\omega'$ can be written as 
\begin{equation}
\omega'=\mu\overline{\partial\gamma'^{0,1}}+\overline{\mu}\partial\gamma'^{0,1}+i\partial\bar\partial\phi'
\end{equation}
where $\mu\in {\mathbb C}$ and $\phi'\in {\mathcal C}^{\infty}(X',{\mathbb R})$. We need to show that $\phi'$ is the pull-back of a ${\mathcal C}^{\infty}$ function $\phi$ on $X$.

Locally on a disk $\Delta^2=\{\vert z\vert <1\}$ around $p$ on $X$, $\gamma^{0,1}$ is $\bar\partial$ exact, so it can be written
as $\gamma^{0,1}\vert\Delta^2=\bar\partial f$, where $f\in {\mathcal C}^{\infty}(\Delta^2)$. Then on $\pi^{-1}(\Delta^2)$,
\begin{equation}
\omega'=i\partial\bar\partial (2{\rm Im}(\bar\mu f')+\phi')
\end{equation}
where $f'=\pi^*f$. Set $\varphi'=2Im(\bar\mu f)+\phi'$. We need to show that $\varphi'$ is the pull-back of a smooth function on $\Delta^2$. 

So let $\pi:\hat{\Delta}^2\to\Delta^2$ be the blow-up of the unit disk in ${\mathbb C}^2$, let $E$ be the exceptional divisor, and suppose that locally $\pi$ is given by $(z,w)\to (z,zw)=(z_1,z_2)$. The exceptional divisor is given by $\{z=0\}$. Let $\varphi'$ be a ${\mathcal C}^{\infty}$ function on $\hat{\Delta}^2$. Then we have

\begin{prop}
There exists $\varphi$ a ${\mathcal C}^\infty$ function on $\Delta^2$ such that $\varphi'=\pi^*\varphi$ if and only if there
exist $A_{\alpha,\beta}^{p,q}\in {\mathbb C}$ such that 
\begin{equation}\label{system1}
{\frac{\partial^{\alpha+\beta}\varphi'}{\partial z^{\alpha}\partial\bar z^{\beta}}\vline}_{z=0} =\sum
_{p=0}^{\alpha}\sum_{q=0}^{\beta}
\binom{\alpha}{p}\binom{\beta}{q}A_{\alpha,\beta}^{p,q}w^p\bar w^q
\end{equation}
\end{prop}
\begin{proof}
If $\varphi'=\pi^*\varphi$, with $\varphi\in{\mathcal C}^{\infty}(\Delta^2)$, then, from $\varphi'(z,w)=\varphi (z, zw)$ and the chain rule, we obtain the above equation with 
\begin{equation}
A_{\alpha,\beta}^{p,q}=\frac{\partial^{\alpha+\beta}\varphi}{\partial z_1^p\partial z_2^{\alpha -p}\partial\bar z_1^q\partial\bar z_2^{\beta-q}}(0)
\end{equation}
Conversely, if $\varphi'$ satisfies the above conditions on its partial derivatives, then $\varphi'\vert_{E}$ is constant, and it induces a continuous function $\varphi$ on $\Delta^2$. It is actually ${\mathcal C}^{\infty}$, with the partial derivatives at $0$ equal to $A_{\alpha,\beta}^{p,q}$ as above. 
\end{proof}

\begin{remark} If the above equation \ref{system1} holds only for $\alpha+\beta \leq k$, it follows that $\varphi'$ is the pull-back of a ${\mathcal C}^k$ function $\varphi$.
\end{remark}

So in order to prove that $\varphi'$ is the pull-back of a ${\mathcal C}^{\infty}$ function $\varphi$ on $\Delta^2$, it is enought to prove that 

\begin{equation}
{\frac{\partial^{\alpha+\beta}\varphi'}{\partial z^{\alpha}\partial\bar z^{\beta}}\vline}_{z=0}
\end{equation}
are polynomials in $w$ and $\bar w$ of degrees $\alpha$ and $\beta$ respectively.

\subsection{The system of differential equations} Now we set up the system of differential equations which needs to be solved in order to prove that $\omega'$ is the pull-back of a smooth $\omega$.

We will use the fact that $\omega'$ is of rank $1$ (\cite{blaine}, Proposition 37), i. e., that
\begin{equation}
\omega'\wedge\omega'=0
\end{equation}
and we will show that $\varphi$ is of class ${\mathcal C}^3$, i. e., that $\omega'$ is the
pull-back of a ${\mathcal C}^1$ form.

First, $\omega'=i\partial\bar\partial\varphi'$ and it is positive, hence $\varphi'$ is plurisubharmonic.
Restricted to the exceptional divisor $E$, it follows that $\varphi'\vert_E$ is constant. Hence
$\varphi'$ is the pull-back of a continuous function $\varphi$ on $\Delta^2$.

Next, denote by 
\begin{equation}
P_{\alpha,\beta}={\frac{\partial^{\alpha+\beta}\varphi'}{\partial z^{\alpha}\partial\bar z^{\beta}}\vline}_{z=0}
\end{equation}
which are ${\mathcal C}^{\infty}$ functions on ${\mathbb C}$. Since $\varphi'$ is defined on the whole $\hat{\Delta}^2$,
the functions $P_{\alpha,\beta}$ satisfy the following {\it growth conditions}:
\begin{equation}\label{growth}
w^{\alpha}\bar w^{\beta}P_{\alpha,\beta}\left(\frac 1w\right )
\end{equation}
can be extended to ${\mathcal C}^{\infty}$ functions at $0$.

Consider the equation $\omega'\wedge\omega'=0$ written in local coordinates $(z,w)$:
\begin{equation}
\frac{\partial^2\varphi'}{\partial z\partial\bar z}\cdot \frac{\partial^2\varphi'}{\partial w\partial\bar w}=
\frac{\partial^2\varphi'}{\partial z\partial\bar w}\cdot\frac{\partial^2\varphi'}{\partial w\partial\bar z}
\end{equation}
Take 
\begin{equation}\frac{\partial^{\alpha+\beta}}{\partial z^{\alpha}\partial\bar z^{\beta}}
\end{equation}
and restrict it to $z=0$; we obtain

\begin{equation*}
\sum_{p=0}^{\alpha}\sum_{q=0}^{\beta}\binom{\alpha}{p}\binom{\beta}{q}P_{p+1,q+1}\frac{\partial^2P_{\alpha-p,\beta-q}}{\partial w\partial\bar w}=
\end{equation*}
\begin{equation}\label{system}
=\sum_{p=0}^{\alpha}\sum_{q=0}^{\beta}\binom{\alpha}{p}\binom{\beta}{q}\frac{\partial P_{p+1,q}}{\partial\bar w}\frac{\partial P_{\alpha-p,\beta-q+1}}{\partial w}
\end{equation}
which gives a system of partial differential equations in the unknowns $P_{\alpha,\beta}$ which satisfy the conditions \ref{growth} and moreover $\overline{P}_{\alpha,\beta}=P_{\beta,\alpha}$.

We know that $P_{0,0}$ is constant, and from 
\begin{equation}
P_{1,1}\cdot\frac{\partial^2 P_{0,0}}{\partial w\partial\bar w}=\frac{\partial P_{1,0}}{\partial\bar w}\cdot\frac{\partial P_{0,1}}{\partial w}
\end{equation}
we obtain that $P_{1,0}$ is holomorphic, and from the growth condition \ref{growth} it follows that $P_{1,0}$ has the desired form, i. e., it is a polynomial in $w$ of degree $1$. This shows that $\varphi$ is a function of class ${\mathcal C}^1$.


\subsection{The proof of Theorem \ref{system2}} We complete the proof of theorem \ref{system2}. We show that $\varphi$ is in fact of class ${\mathcal C}^3$, hence $\omega$ is of class ${\mathcal C}^1$.

For $\alpha=2$ and $\beta=0$ in \ref{system} we obtain

\begin{equation}
P_{1,1}\cdot\frac{\partial^2P_{2,0}}{\partial w\partial\bar w}=2\frac{\partial P_{2,0}}{\partial\bar w}\cdot\frac{\partial P_{1,1}}{\partial w}
\end{equation} 
and for $\alpha=1,\beta=1$ we obtain

\begin{equation}
P_{1,1}\cdot\frac{\partial^2P_{1,1}}{\partial w\partial\bar w}=\frac{\partial P_{2,0}}{\partial\bar w}\cdot\frac{\partial P_{0,2}}{\partial w}+\frac{\partial P_{1,1}}{\partial\bar w}\cdot\frac{\partial P_{1,1}}{\partial w}
\end{equation}

Set $$f=\frac{\partial P_{2,0}}{\partial\bar w}$$ and $g=P_{1,1}$. Then $f$ and $g$ satisfy the following properties: they are ${\mathcal C}^{\infty}$ functions on ${\mathbb C}$; $g$ has real values; the functions

\begin{equation}\label{growth1}
w\bar w g\left(\frac 1w\right)
\end{equation}

and

\begin{equation}\label{growth2}
\frac{w^2}{\bar w^2}\cdot f\left(\frac 1w\right)
\end{equation}

are ${\mathcal C}^{\infty}$ at $0$, and moreover $f$ and $g$ satisfy the following equations:

\begin{equation}\label{1eq}
\frac{\partial f}{\partial w}\cdot g=2 f\cdot\frac{\partial g}{\partial w}
\end{equation}
\begin{equation}\label{2eq}
g\cdot\frac{\partial^2g}{\partial w\partial\bar w}=\vert f\vert^2+\left|\frac{\partial g}{\partial w}\right|^2
\end{equation}

We will show the following

\begin{prop}\label{radu}
 $f=0$ and $g$ is a quadratic form of rank $1$, i. e., $g(w)=|a+bw|^2$.
\end{prop}

\begin{proof}
Let $D_g$ be the non-zero set of $g$, i. e., $D_g=\{w\in {\mathbb C}|g(w)\neq 0\}$. If $D_g=\emptyset$, then
$g=0$ and from \ref{2eq} it follows that $f=0$. 

If $D_g={\mathbb C}$, then $g$ is never $0$, and from \ref{1eq} it
follows that there exists $h$ holomorphic on ${\mathbb C}$ such that $f=\bar h g^2$. We can assume that $g>0$ on
${\mathbb C}$. Then from \ref{2eq} it follows that $\ln g$ is subharmonic, hence $\ln |f|$ is subharmonic on $D_f=\{w\in {\mathbb C}|f(w)\neq 0\}$. It follows that $|f|^2$ is subharmonic on ${\mathbb C}$ and since
$f$ is bounded (from \ref{growth2}), it follows that $|f|$ is constant.
If $|f|\neq 0$, then from $f=\bar h g^2$ we obtain that $i\partial\bar\partial\ln g=0$ and from \ref{2eq} we get that $|f|=0$, contradiction. Hence $f=0$ and equation \ref{2eq} implies that $\ln g$ is harmonic, i. e., $g=\exp({\rm Re} j)$, where $j$ is a holomorphic function on ${\mathbb C}$. From condition \ref{growth1} on $g$ it follows that $j$ is constant, hence also $g$ is constant.

Now assume that $D_g\neq \emptyset, {\mathbb C}$ and denote by $D_g'$ a connected component of $D_g$. Assume that $g>0$ on $D_g'$. From \ref{1eq} it follows that $f=\bar h g^2$ where $h$ is a holomorphic function on $D_g'$. Again \ref{2eq} implies that $\ln g$ is subharmonic on $D_g'$ and so $\ln |f|$ is subharmonic on $D_g'\cap D_f$. Let $w_0\in\partial D_g'$ (the boundary of $D_g'$) and  set 
\begin{equation}
f'(w)=\frac{f(w)}{\sqrt{|w-w_0|}}
\end{equation}
 as a function on $D_g'$. Since $\ln|f|$ is subharmonic, it follows that $\ln|f'|$ is also subharmonic on $D_g'$, so $|f'|^2$ is subharmonic on $D_g'$. Moreover, $f=0$ on the boundary $\partial D_g'$ (this follows again from \ref{2eq}) except possibly at $w_0$, and $\lim_{w\to \infty} |f'(w)|=0$ because $f$ is bounded at infinity (from \ref{growth2}). Since $f(w_0)=0$ it follows that $f'$ can be extended to a continuous function at $w_0$, with $f'(w_0)=0$. Hence $|f'|$ is a subharmonic function on $D_g'$, $f'=0$ on $\partial D_g'\cup\{\infty\}$, hence from the maximum principle, it follows that $f'=0$ on $D_g'$, hence also $f=0$ on $D_g'$. Since $f=0$ on $\{w\in {\mathbb C}|g(w)=0\}$, we get that $f=0$ on the whole ${\mathbb C}$. 

So $g$ satisfes the equation 
\begin{equation}
g\cdot\frac{\partial^2 g}{\partial w\partial\bar w}=\frac{\partial g}{\partial w}\cdot\frac{\partial g}{\partial\bar w}
\end{equation}
 and 
 \begin{equation}
w\bar w\cdot g\left ( \frac 1w\right )
\end{equation}
is ${\mathcal C}^{\infty}$ at $0$. If $g$ has two zeroes, $w_0$ and $w_1$, $w_0\neq w_1$, we consider as above $D_g'$ a connected component of $D_g$. Assume that $g>0$ on $D_g'$. Then $\ln g$ is harmonic on $D_g'$. Let $$g'(w)=\frac{g(w)}{\sqrt{|w-w_0|^3}\sqrt{|w-w_1|^3}}$$
Then $\ln g'$ is harmonic on $D_g'$, so $g'$ is subharmonic. Moreover, it is $0$ on the boundary $\partial D_g'$ of $D_g'$, except possibly at $w_0$ and $w_1$. But at $w_0$, $g(w_0)=0$ and 
\begin{equation}
\frac{\partial g}{\partial w}(w_0)=\frac{\partial g}{\partial \bar w}(w_0)=0
\end{equation} 
and the same at $w_1$, which implies that $g'$ is continuous on the whole boundary $\partial D_g'$.
At infinity, $g$ approaches $0$, and again by the maximum principle we obtain that $g=0$ on $D_g'$, contradiction. This shows that $g$ has exactly one zero. Assume that $g(w_0)=0$. Then consider the function 
$$g''(w)=\frac{g(w)}{|w-w_0|^2}$$
on ${\mathbb C}\setminus \{w_0\}$. Then $\ln g''$ is harmonic on ${\mathbb C}\setminus \{w_0\}$, and it is bounded at infinity. Moreover, since $g(w_0)=0$ and $dg(w_0)=0$, it follows that $g''$ is bounded near $w_0$. Hence $g''$ is a bounded, subharmonic function on ${\mathbb C}\setminus \{w_0\}$, so it is constant. Therefore $g(w)=C|w-w_0|^2$.
\end{proof}

Returning to our previous notations, we showed that $P_{2,0}$ is holomorphic, hence it is a polynomial of degree $2$ in $w$, and that $P_{1,1}$ is a polynomial of degree $\leq 1$ in $w$ and $\bar w$. Hence $\varphi$ is a function of class ${\mathcal C}^2$ and $\omega$ is continuous.

Next, we show that if $P_{1,1}\neq 0$, then $\varphi$ is actually 
${\mathcal C}^{3}$. 
First, we can assume, without loss of generality, that $P_{1,1}$ is constant. Indeed, if $P_{1,1}(w)=C|w-w_0|^2$, then we replace the functions $P_{\alpha,\beta}$ by 
\begin{equation}
\frac{1}{(w-w_0)^{\alpha}(\bar w-\bar w_0)^{\beta}}P_{\alpha,\beta}(w)
\end{equation}
and we end up with the same system of differential equations and the same {\it growth conditions}. 

When $\alpha=3$ and $\beta=0$ in \ref{system} we obtain 
\begin{equation}
P_{1,1}\cdot\frac{\partial ^2 P_{3,0}}{\partial w\partial \bar w}=3\cdot \frac{\partial P_{3,0}}{\partial\bar w}\cdot\frac{\partial P_{1,1}}{\partial w}
\end{equation}
and when $\alpha=2$ and $\beta=1$ we obtain
\begin{equation}
P_{1,1}\cdot\frac{\partial^2P_{2,1}}{\partial w\partial\bar w}+2\cdot P_{2,1}\cdot\frac{\partial P_{1,1}}{\partial w\partial\bar w}=\frac{\partial P_{1,1}}{\partial \bar w}\cdot\frac{\partial P_{2,1}}{\partial w}+2\cdot\frac{\partial P_{2,1}}{\partial\bar w}\cdot\frac{\partial P_{1,1}}{\partial w}
\end{equation}
$P_{1,1}$ is a non-zero constant, so the equations imply that both $P_{3,0}$ and $P_{2,1}$ are harmonic. By using the {\it growth conditions} we obtain that $P_{3,0}$ is holomorphic and that $P_{2,1}$ has the desired form.

If $P_{1,1}=0$, things get more complicated, but we can still show that $\varphi$ is of class ${\mathcal C}^3$. If $\omega (0)=0$, then for $\alpha+\beta=4$ the system \ref{system} implies the following equations:

\begin{equation}
3f\cdot\frac{\partial g}{\partial w}=2\frac{\partial f}{\partial w}\cdot g
\end{equation}

\begin{equation}
\bar g\cdot\frac{\partial f}{\partial w}+3g\cdot\frac {\partial^2 g}{\partial w\partial \bar w}=3\frac{\partial g}{\partial w}\cdot\frac{\partial g}{\partial\bar w}+3f\frac{\partial\bar g}{\partial w}
\end{equation}

\begin{equation}
2g\cdot\frac{\partial^2\bar g}{\partial w\partial\bar w}+2\bar g\frac{\partial^2 g}{\partial w\partial\bar w}=
\frac{\partial g}{\partial w}\cdot\frac{\partial\bar g}{\partial\bar w}+4\frac{\partial g}{\partial \bar w}\cdot\frac{\partial\bar g}{\partial w}+f\cdot\bar f
\end{equation}
where $$f=\frac{\partial P_{3,0}}{\partial\bar w}$$ and $g=P_{2,1}$ and we have the 
corresponding {\it growth conditions} for $f$ and $g$. This system can be solved by using similar methods as in Proposition \ref{radu}, so we omit it.

\section{Manifolds of maximal K\"ahler rank}

In this section we show that a compact complex manifold $X$ of dimension $n$ such that $Kr(X)=n$ and which moreover admits a special Hermitian metric is in fact K\"ahler:

\begin{thm}
Let $X$ be a compact complex manifold such that there exists a {\it nef} class $\{\alpha\}\in H^{1,1}_{BC}(X,{\mathbb R})$ such that $$\int_X\alpha^n>0$$ Suppose moreover that $X$ supports a Hermitian metric $h$ such that 
\begin{equation}\label{compatibility}
i\partial\bar\partial h=\partial h\wedge\bar\partial h =0
\end{equation}
Then $\{\alpha\}$ is {\it big} and $h$ is $\partial+\bar\partial$ cohomolgous to a K\"ahler metric. In particular $X$ is K\"ahler.
\end{thm}
\begin{remark}
Here {\it big} means that the class $\{\alpha\}$ contains a K\"ahler current, i.e., a closed positive current that dominates some Hermitian metric.
\end{remark}
\begin{remark}
Condition \ref{compatibility} is needed in order to bound some integrals (see \ref{integralinequality} below) and it is equivalent to 
\begin{equation}
i\partial\bar\partial h^k=0,\forall k=\overline{1,n-1}
\end{equation}
The condition \ref{compatibility} appeared in the work \cite{guanli}, where the authors attempted to solve the Monge-Amp\`ere equation on Hermitian manifolds.
\end{remark}
\begin{remark}\label{kahsurf}
When $n=2$, the existence of a Hermitian form satisfying  \ref{compatibility} is well-known, and we obtain another proof of the fact that a surface of K\"ahler rank equal to $2$ is K\"ahler. When $n=3$ just the equation $i\partial\bar\partial h=0$ is needed.
\end{remark}
\begin{remark}
The above theorem is a particular case of a conjecture of Demailly and P\u aun (see \cite{demaillypaun}, Conjecture 0.8) which states that if a manifold admits a nef class of strictly positive self-intersection, the the manifold is in Fujiki class ${\mathcal C}$, i.e., it is bimeromorphic to a K\"ahler manifold.
\end{remark}
\begin{proof}
First, we show that $\{\alpha\}$ is big. We need to show that there exists $\varepsilon_0>0$ and a distribution $\chi$ such that $\alpha+i\partial\bar\partial\chi\geq \varepsilon_0 h$. According to Lamari's result \cite{lamari}, Lemme 3.3, this is equivalent to showing that
\begin{equation}
\int_X\alpha\wedge g^{n-1}\geq \varepsilon_0\int_X h\wedge g^{n-1}
\end{equation}
 for any Gauduchon metric $g^{n-1}$ on $X$. So suppose that $\forall m\in {\mathbb N},\exists g_m^{n-1}$ a Gauduchon metric such that 
\begin{equation}
\int_X\alpha \wedge g_m^{n-1} \leq\frac 1m\int_X h\wedge g_m^{n-1}
\end{equation}
We can assume that 
\begin{equation}
\int_X h\wedge g_m^{n-1}=1 
\end{equation}
and therefore 
\begin{equation}
\int_X \alpha\wedge g_m^{n-1}\leq \frac 1m
\end{equation} 
Since $\{\alpha\}$ is nef, for every $m$ we can find 
$\psi_m\in {\mathcal C}^{\infty}(X,{\mathbb R})$ 
such that $\alpha+i\partial\bar\partial\psi_m\geq -\frac {1}{2m}h$. The main result of \cite{tossatiw} implies that we can solve the equation
\begin{equation}\label{monge}
\left (\alpha+\frac 1m h+i\partial\bar\partial\varphi_m\right )^n=C_m g_m^{n-1}\wedge h
\end{equation}
for a function $\varphi_m\in{\mathcal C}^{\infty}(X,{\mathbb R})$ such that if we set 
$\alpha_m=\alpha+\frac 1m h+i\partial\bar\partial\varphi_m$, then $\alpha_m>0$. The constant $C_m$ is given by
\begin{equation}
C_m=\int_X\left (\alpha+\frac 1m h\right )^n\geq\int_X\alpha^n=C>0
\end{equation}
Now 
\begin{equation}\label{integralinequality}
\int_X\alpha_m^{n-1}\wedge h=\int_X h\wedge \left(\alpha+\frac 1m h\right )^{n-1}\leq \int_Xh\wedge \left(\alpha+h\right )^{n-1}=M
\end{equation}
and if we set 
\begin{equation}
E=\left \{\frac{\alpha_m^{n-1}\wedge h}{g_m^{n-1}\wedge h}> 2M\right \}
\end{equation}
then 
\begin{equation}\label{edelta}
\int_{E}g_m\wedge h\leq \frac 12
\end{equation}
Therefore on $X\setminus E$ we have $\alpha_m^{n-1}\wedge h\leq 2M g_m^{n-1}\wedge h$. By looking at the eigenvalues of $\alpha_m$ with respect to $h$, from \ref{edelta} and \ref{monge}, it follows that on $X\setminus E$ we have $$\alpha_m\geq \frac{C_m}{2nM}h$$ Therefore

\begin{equation}\label{ineq}
\int_X \alpha_m\wedge g_m^{n-1}\geq\int_{X\setminus E}\alpha_m\wedge g_m^{n-1}\geq \frac {C_m}{2nM}\int_{X\setminus E}h\wedge g_m^{n-1}=
\end{equation}

\begin{equation*}
=\frac{C_m}{2nM}\left (\int_Xh\wedge g_m^{n-1}-\int_{E}h\wedge g_m^{n-1}\right )\geq \frac {C}{4nM} 
\end{equation*}
On the other hand
\begin{equation}
\int_X\alpha_m\wedge g_m^{n-1}=\int_X\alpha\wedge g_m^{n-1}+\frac 1m\int_X h\wedge g_m^{n-1}\leq \frac 2m
\end{equation}
contradiction with \ref{ineq}.

Therefore $\{\alpha\}$ is big, and from \cite{demaillypaun} it follows that $X$ is in the Fujiki class ${\mathcal C}$. Theorem 2.2 in \cite{chiose} implies that a manifold in the Fujiki class ${\mathcal C}$ and which is $SKT$
(strong K\"ahler with torsion, i.e., it supports a $i\partial\bar\partial$-closed Hermitian metric), is in fact K\"ahler.
\end{proof}

\begin{remark}
A very similar method gives a much simpler proof of a key result in \cite{demaillypaun} Theorem 0.5 that a nef class on a compact K\"ahler manifold of strictly positive self-intersection 
contains a K\"ahler current. Indeed, suppose 
$\{\alpha\}$ is not big, then by Lamari \cite{lamari} 
there exists a sequence of Gauduchon metrics such that 
$$\int_X\alpha\wedge g_m^{n-1}\leq \frac 1m $$ and 

$$\int_Xh\wedge g_m^{n-1}=1$$ If $h$ is assumed to be K\"ahler, the proof proceeds as above to obtain a contradiction. This proof is not independent of the proof of Demailly and P\u aun. In a few words, we replaced the explicit and involved construction of the metrics $\omega_{\varepsilon}$ in \cite{demaillypaun} by the abstract sequence of Gauduchon metrics given by the Hahn-Banach theorem, via Lamari \cite{lamari}
\end{remark}

\begin{remark}
An adaptation of the proof of Theorem 0.5 in \cite{demaillypaun} can not work in our case. One of the obstructions is that, if a complex manifold $X$ admits a Hermitian metric with property \ref{compatibility}, then it is not clear that $X\times X$ admits a Hermitian metric with the same property.
\end{remark}

\begin{remark}
We should also point out that a simplified proof of another part of the proof of the Demailly and P\u aun theorem on the K\"ahler cone was given recently by Collins and Tosatti \cite{collinstosatti}. Together with the above proof, one obtains a more compact proof of the main result in \cite{demaillypaun}
\end{remark}


 \providecommand{\bysame}{\leavevmode\hbox
to3em{\hrulefill}\thinspace}

Address:
\par\noindent{\it Ionu\c t Chiose}: \\
Institute of Mathematics of the Romanian Academy\\
P.O. Box 1-764, Bucharest 014700\\
Romania\\
{\tt Ionut.Chiose@imar.ro}\\


\begin{thebibliography}{Dyn52b}


\bibitem[BHPV]{bpv} {W. Barth, K. Hulek, C. Peters, A. Van de Ven}:
{\em Compact complex surfaces}, Ergebnisse der Mathematik und ihrer Grenzgebiete, 
Berlin, Springer-Verlag,  (2004).

\bibitem[Br]{brunella} {M. Brunella}, {\em A characterization of Inoue surfaces}, to appear in Comm. Math. Helv.


\bibitem[ChTo]{chiosetoma} {I. Chiose, M. Toma}, {\em On compact complex surfaces of K\"ahler rank one}, Amer. J. Math., {\bf 135} (2013), no. 3,  851--860

\bibitem[Ch]{chiose} {I. Chiose}, {\em Obstructions to the existence of K\"ahler structures on compact complex manifolds}, to appear in Proc. Amer. Math. Soc.

\bibitem[CoTo]{collinstosatti} {T. C. Collins, V. Tosatti}, {\em K\"ahler currents and null loci}, arXiv:1304.5216


\bibitem[DeP\u a]{demaillypaun} {J.-P. Demailly, M. P\u aun}, {\em Numerical characterization of the K\"ahler cone of a compact K\"ahler manifold}, Annals of Math. 159 (2004) 1247-1274 

\bibitem[FLY]{fuliyau} {J. Fu, J. Li, S.-T. Yau}, {\it
Balanced metrics on non-K\"ahler Calabi-Yau threefolds}
, J. Diff. Geom. {\bf 90} (2012), 81--130

\bibitem[GuLi]{guanli} {B. Guan, Q. Li}, {\em Complex Monge-Amp\`ere equations and totally real submanifolds}, Adv. Math. {\bf 225} (2010), no. 3, 1185–-1223 

\bibitem[HaLa]{blaine} {R. Harvey, H. B. Lawson Jr.}, {\em
An intrinsic characterization of K\" ahler manifolds}, Invent. Math.
{\bf 74} (1983), no. 2, 169--198

\bibitem[Hi]{hironaka} {H. Hironaka}, {\em An example of a non-K\"ahlerian complex-analytic deformation
of K\"ahlerian complex structures}, Ann. of Math. (2) {\bf 75} 1962 190--208

\bibitem[La]{lamari} {A. Lamari}, {\em Courants k\"ahl\'eriens et surfaces compactes},
Ann. Inst. Fourier (Grenoble) {\bf 49} (1999), no. 1, 263–-285

\bibitem[Og]{oguiso} {K. Oguiso}, {\em Two remarks on Calabi-Yau Moishezon treefolds}, J. Reine Angew. Math. {\bf 452} (1994), 153--161





\bibitem[ToWe]{tossatiw} {V. Tosatti, B. Weinkove}, {\em The complex Monge-Ampere equation on compact Hermitian manifolds}, J. Amer. Math. Soc. {\bf 23} (2010), no.4, 1187--1195

\end{thebibliography}
\end{document}